\documentclass{article}
\usepackage[english]{babel}
\usepackage{authblk}
\usepackage[letterpaper,top=2cm,bottom=2cm,left=3cm,right=3cm,marginparwidth=1.75cm]{geometry}

\usepackage{amsmath,amssymb}
\usepackage{graphicx}
\usepackage[colorlinks=true, allcolors=blue]{hyperref}
\newtheorem{theorem}{Theorem}[section]
\newtheorem{lemma}{Lemma}[section]
\newtheorem{proposition}{Proposition}[section]
\newtheorem{corollary}{Corollary}[section]
\newtheorem{definition}{Definition}[section]

\newtheorem{proof}{Proof}[section]
\title{Clifford wavelet transform and the associated Donoho-Stark’s uncertainty Principle}
\author[1,2]{Sabrine Arfaoui}
\affil[1]{Laboratory of Algebra, Number Theory and Nonlinear Analysis,\newline Department of Mathematics, Faculty of Sciences, University of Monastir,\newline Avenue of the Environment, 5019 Monastir, Tunisia.}
\affil[2]{Department of Mathematics, Faculty of Science, University of Tabuk,\newline King Faisal Road, 47512 Tabuk, Saudi Arabia.}
\begin{document}
\maketitle
\begin{abstract}
This paper focuses on studying the Donoho-Stark's type uncertainty principle for the continuous Clifford wavelet transform. A brief review of Clifford algebra/analysis, Clifford wavelet transform and their properties is conducted. Next, such concepts are applied to develop an uncertainty principle based on Clifford wavelets.

\end{abstract}
\maketitle


\section{Introduction} 
In 1927 Heisenberg announced the famous principle of uncertainty, it is one of the most famous and important concepts of quantum mechanics. The physical origin of uncertainty principle is related to quantum systems and states that we cannot know both the position and speed of a particle with perfect accuracy, the more we nail down the particle's position, the less we know about its speed and vice versa. 

Mathematically, an uncertainty principle (UP) is an inequality expressing limitations on the simultaneous concentration of a function and its Fourier transform. According to the mathematical meaning given to the intuitive notion of concentration and to the type of representation chosen for the signal, many different forms of UPs are possible  and, starting from the classical works of Heisenberg, a vast literature is today available on these topics, see \cite{Banouh},  \cite{Dahkleteal}, \cite{ElHaouietal}, \cite{Fu2015}, \cite{Hitzer2}, \cite{Hitzer-Mawardi-1}, \cite{Jday2018}, \cite{Kouetal}, \cite{Mawardi-Ryuichi}, \cite{Mawardi-Ryuichi-1}, \cite{Rachdi-Herch} \cite{Mawardi-Ryuichi-2}, \cite{Mawardi-Hitzer-1}, \cite{Mawardi-Hitzer-2}, \cite{Mawardi-Hitzer-3}, \cite{Mawardi-Hitzer-4}, \cite{Mawardi-Hitzer-Hayashi-Ashino}, \cite{Msehli-Rachdi-1}, \cite{Msehli-Rachdi-2}, \cite{Rachdi-Meherzi}, \cite{Mawardi-Hitzer-Hayashi-Ashino} \cite{Rachdi-Amri-Hammami}, \cite{Rachdi-Herch}, \cite{Soltani}, \cite{Yangetal1}, \cite{DeBie1}, \cite{DeBie2}. 

The qualitative UP is a kind of UPs, which tells us how a signal $f$ and its Fourier transform $\hat{f}$, behave under certain conditions. One such example can be Donoho–Stark’s UP , which expresses the limitations on the simultaneous concentration of $f$, and  $\hat{f}$. 

In this paper we are concerned with the Donoho-Stark form of the uncertainty principle. We propose more precisely to extend the Donoho-Stark uncertainty principle to the context of Clifford algebras by applying the so-called Clifford wavelets. Clifford wavelets or wavelets on Clifford algebras are the last variants of wavelets. The flexibility of Clifford algebras permitted to involve different forms of vector analysis in the same time. For more details on Clifford wavelet transformations, we refer to \cite{Arfaoui1, Arfaoui2, Arfaoui3, Arfaoui4, Arfaoui-Rezgui-book, Arfaouietal21BookWavelet}.

As Donoho-Stark UP relies on the concept of $\epsilon$-concentration of a function on a measurable set $U \subseteq \mathbb{R}^m$. We start by recalling this definition followed by the statement of the classical theorem.
\begin{definition}
Given $\epsilon_U\geq 0$, a function $f$ is  said to be $\epsilon_U$-concentrated in the $L^{2}$-norm on a measurable set $U\subseteq \mathbb{R}^m$ if
\begin{equation}
 ||f(x)||_{L^2(\mathbb{R}^m| U)}  \leq \epsilon_U ||f||_2.
\end{equation}
\end{definition}
\begin{theorem}
(UP Donoho-Stark) Suppose that $f\in L^2(\mathbb{R}^m), f\neq  0$ is $\epsilon_T$-concentrated on $T \subseteq \mathbb{R}^m$ and $\hat{f}$ (Fourier transform of $f$) is $\epsilon_\Omega$-concentrated on $\Omega \subseteq \mathbb{R}^m$, with $T, \Omega$ two measurable sets in $\mathbb{R}^m$ and $\epsilon_T, \epsilon_\Omega > 0, \quad\epsilon_T+\epsilon_\Omega < 1$. Then
\begin{equation}
|T||\Omega| > (1-\epsilon_T-\epsilon_\Omega)^2.
\end{equation}
\end{theorem}

This paper is organized as follows. Section 2 is devoted to a reminder of the basics of Clifford algebras/ analysis. In Section 3, we recall some results and properties for the Clifford-wavelet transform useful. In section 4, we prove the  Donoho-Stark’s uncertainty principle for the  Clifford wavelet transform. Finally, we give a conclusion in Section 5.

\section{Clifford Analysis Revisited} 
In this section a mathematical review of Clifford algebra/analysis is discussed. We introduce the real Clifford algebra $\mathbb{R}_{m}$ over $R^m, (m\geq1)$ as a non commutative algebra with dimension $2^{m}$ generated by the basis $(e_1,\dots, e_m)$ and satisfying:
$$
\begin{cases}
e_j^2=-1,\quad j=1,\dots,m\\
e_je_k+e_ke_j=0,\quad j\neq k,\quad j,k=1,\dots,m.
\end{cases}
$$
Denote next, for $k\in\mathbb{N}$, $\mathbb{R}_{m}^{k}$ the space of $k$-multi vectors defined by
$$
\mathbb{R}_{m}^{k}=span_{\mathbb{R}}\left\{e_{A}=e_{i_{1}i_{2}\dots i_{k}},\,A=(i_1,i_2,\dots,i_k),\;1\leq i_{1}<\cdots<i_{k}\leq m\right\}
$$
where $e_{i_{1}i_{2}\dots i_{k}}=e_{i_{1}}e_{i_{2}}\cdots e_{i_{k}}$ and $e_{\emptyset}=1$. 

The Clifford algebra $\mathbb{R}_{m}$ may be decomposed as a direct sum
$$
\mathbb{R}_{m}=\bigoplus_{k=0}^{m}\mathbb{R}_{m}^{k}.
$$

Any element $u$ of the Clifford algebra $\mathbb{R}_{m}$, is called multivector, and can be expressed in the form 
$$
u=\sum_{A}u_{A}e_{A}.
$$
Let $|A|$ is the length of the multi-index $A$, $a$ may be written as
$$
u=\sum_{k=0}^{n}\sum_{\left|A\right|=k}u_{A}e_{A},\;\;u_{A}\in\mathbb{R}.
$$

The conjugation defined on the basis $(e_A)_{A}$ on the  Clifford algebra $\mathbb{R}_{m}$ is  
$$
\overline{e_{A}}=(-1)^{\frac{|A|(|A|+1)}{2}}e_{A}, \,\forall A.
$$

It is a non commutative operation for which we have $\overline{uv}=\overline{v}\overline{u},\, \forall u,v\in\mathbb{R}_{m}$.

These concepts of real Clifford algebras may be extended to the complex Clifford algebra $\mathbb{C}_{m}=\mathbb{R}_{m}+i\mathbb{R}_{m}$. In this case any element $\lambda\in\mathbb{C}_{m}$ may be decomposed as $$
\lambda=u+iv=\displaystyle\sum_{A}\lambda_{A}e_{A},\,\lambda_{A}\in\mathbb{C},\, u,v\in\mathbb{R}_m
$$
and where the basis $(e_j)_{1\leq j\leq m}$ is this way an orthonormal (canonical) basis of $\mathbb{C}_m$. This induces an involution on $\mathbb{C}_m$ known as the hermitian conjugation $$
\lambda^{\dagger}=\overline{a}-i\overline{b}.
$$

A vector $x=(x_{1},x_{2},\dots,x_{m})\in\mathbb{R}^{m}$ may be identified to the Clifford element in $\mathbb{R}_{m}$, 
$$
\underline{x}={\displaystyle \sum_{j=1}^{m}x_{j}e_{j}}.
$$
We define the Clifford product of two vectors by
$$
\underline{x}\underline{y}=\underline{x}\bullet\underline{y}+\underline{x}\wedge\underline{y},
$$
where the $\bullet$ product is an analogous of the classical inner product on $\mathbb{R}^{m}$,
$$
\underline{x}\bullet\underline{y}=-<\underline{x},\underline{y}>=-{\displaystyle \sum_{j=1}^{m}x_{j}y_{j}},
$$
and where the $\wedge$ product is the outer product
$$
\underline{x}\wedge\underline{y}={\displaystyle\sum_{j<k}e_{j}e_{k}(x_{j}y_{k}-x_{k}y_{j})}.
$$
This yields that
$$
\underline{x}\bullet\underline{y}=\frac{1}{2}(\underline{x}\underline{y}+\underline{y}\underline{x})\;\;\hbox{and}\;\;
\underline{x}\wedge\underline{y}=\frac{1}{2}(\underline{x}\underline{y}-\underline{y}\underline{x}).
$$
In particular we have
$$
\underline{x}^{2}=-\left|\underline{x}\right|^{2}=-\sum_{j=1}^{m}|x_{j}|^{2}.
$$

Now let $\underline{\omega}$ be a vector in $\mathbb{R}^m$ with norm $|\underline{\omega}| = 1$, i.e. $\underline{\omega}$ is an element of the unit sphere $\mathcal{S}^{n-1}\subset \mathbb{R}^m$. Any vector $\underline{x}$  may then be decomposed as a sum of two terms, one being parallel with $\underline{\omega}$ and the other being perpendicular to $\underline{\omega}$:
$$
\underline{x}=\underline{x}_{\parallel\underline{\omega}}+\underline{x}_{\perp\underline{\omega}}= \left\langle \underline{x},\underline{\omega}\right\rangle \underline{\omega} + \underline{\omega}\left(\underline{x}\wedge\underline{\omega}\right)
$$
which permits next to characterize the reflection $R_{\underline{\omega}}$
with respect to the hyperplane $\underline{\omega}^{\perp}$ as 
$$
R_{\underline{\omega}}(\underline{x})=\underline{\omega}\underline{x}\underline{\omega}.
$$

Cartan-Dieudonn\'e Theorem (\cite{Cartan1966}) relates the reflection
to the so-called spinors. Stating that there exists $\underline{\omega}_{1},\underline{\omega}_{2},\dots,\underline{\omega}_{2l}\in\mathcal{S}^{n-1}$ with $\underline{\omega}_{j}^{2}=-1,1\leq j\leq2l$ ($l\in\mathbb{N}$) and a rotation $T\in SO(n)$ ($SO(n)$ being the rotation group on $\mathbb{R}^n$ ) such that
\[
T(\underline{x})=\left[R_{\underline{\omega}_{1}}\circ R_{\underline{\omega}_{2}}\circ......\circ R_{\underline{\omega}_{2l}}\right](\underline{x})=\underline{\omega}_{1}\underline{\omega}_{2},......\underline{\omega}_{2l}\underline{x}\,\underline{\omega}_{2l}\underline{\omega}_{2l-1},......\underline{\omega}_{2}\underline{\omega}_{1}.
\]
Denoting $s=\underline{\omega}_{1}\underline{\omega}_{2},......\underline{\omega}_{2l}$
and $\overline{s}=\underline{\omega}_{2l}\underline{\omega}_{2l-1},......\underline{\omega}_{2}\underline{\omega}_{1}$ we have $T(\underline{x})=s\underline{x}\overline{s}$. The element $s$ is called a spinor. Generally speaking, the spin group of order $n$ is
\[
Spin(n)=\left\{ s\in\mathbb{R}_{n};\;s={\displaystyle \prod\limits _{j=1}^{2l}}\underline{\omega}_{j},\;\underline{\omega}_{j}^{2}=-1,1\leq j\leq2l\right\}.
\]

In Clifford functional framework, a function $f$ defined on the vector space $\mathbb{R}^{m}$ and taking values in the Clifford algebra  $\mathbb{R}_m$ or  $\mathbb{C}_m$  will be expressed as
\begin{equation}\label{eq:Clifford valued func}
f(\underline{x})=\sum_{A}e_{A}f_{A}(\underline{x}),
\end{equation}
where $f_{A}$ are real-valued functions and $A\subset\left\{ 1,2,\cdots,n\right\} $.
Its conjugate $\overline{f}$ is given by
\begin{equation}
 \overline{f}(\underline{x})=\sum_{A}\overline{e_{A}}f_{A}(\underline{x})   
\end{equation}
for a function with values in the “real” Clifford algebra $\mathbb{R}_{m}$. However, the conjugate will be
\begin{equation}
f^{\dagger}(\underline{x})=\sum_{A}e_{A}^{\dagger}f_{A}(\underline{x})^{\dagger}
\end{equation}
for a function with values in the complex Clifford algebra  $\mathbb{C}_{m}$.

 We denote the inner product of functions $f$ and $g$ in the framework of Clifford analysis by

\begin{equation}\label{eq:inner product}
\left\langle f,g\right\rangle_{L^{2}(\mathbb{R}^{m},\mathbb{R}_{m},dV(\underline{x}))} =\int_{\mathbb{R}^{m}}f(\underline{x})^{\dagger}g(\underline{x})dV(\underline{x}).
\end{equation}
and the associated norm by
\begin{equation}
\|f\|_{L^{2}(\mathbb{R}_{m},dV(\underline{x}))} =<f,f>_{L^{2}(\mathbb{R}_{m},dV(\underline{x}))}^{\frac{1}{2}}.
\end{equation}

For more backgrounds on the Clifford algebra/ analysis  the readers may be referred
also to ( \cite{Hamilton1866}, \cite{Cartan1966}, \cite{Brackx-Schepper-Sommen0}.)

\section{Clifford wavelet transform}
Recently, it has become popular to generalize the integral transforms from real and complex numbers to Clifford algebra to study higher dimension such as the Clifford Fourier transform CFT and the Clifford wavelet transform CWT. 

Wavelets are functions that satisfy certain requirements and are used in representing and processing functions and signals, as well as, in
compression of data and images as in fields such as: mathematics, physics, computer science, engineering, and medicine. The study of wavelet transforms had been motivated by the need to overcome some weak points in representing functions and signals by the classical Fourier transforms such as the speed of convergence and Gibbs phenomenon. In addition, wavelet transforms have showed superiority over the classical Fourier transforms. In many applications, wavelet transforms converge faster than Fourier transforms, leading to more efficient processing of signals and data. 

Clifford wavelets or wavelets on Clifford algebras are the last variants of wavelet functions developed by researchers in order to overcome many problems that are not well investigated by classical transforms. The challenging in such concepts is not the wavelet functions themselves but the structure of Clifford algebras and their flexibility to include different forms of vector analysis in the same time.

In the present section we propose to review briefly the two main methods to construct Clifford wavelets. The first one is based on Spin groups and thus includes the factor of rotations in the wavelet analysis provided with the translation and dilatation factors. This method generalizes in some sense the first essay in developing multidimensional wavelets such as Cauchy ones. The second one is concerned with wavelets issued from monogenic functions mainly polynomials. These ones constitute an extension of orthogonal polynomials to the case of Clifford algebras. Recall that orthogonal polynomials are widely applied in wavelet theory on Euclidean spaces. Extensions to the case of Clifford framework are essentially developed by Arfaoui et al \cite{Arfaoui1,Arfaoui2,Arfaoui3,Arfaoui4,Arfaoui-Rezgui-book,Arfaouietal21BookWavelet} and Brackx et al \cite{Brackxetal2013,Brackx2001a}.
\subsection{Clifford wavelets based on the Spin group}
As in the classical cases of wavelets  on Euclidean spaces we seek some properties to be satisfied for a Clifford algebra valued function  to be a mother wavelet.
\begin{definition}Let $\psi\in L^{1}\cap L^{2}(\mathbb{R}^{m},\mathbb{R}_{m},dV(\underline{x}))$. The function $\psi$ is said to be a Clifford mother wavelet  iff the following assertions hold simultaneously.
\begin{itemize}
\item $\widehat{\psi}(\underline{\xi})\left[\widehat{\psi}(\underline{\xi})\right]^{\dagger}$ is scalar.
\item The admissibility  condition
\begin{equation}
\mathcal{A}_{\psi}={\displaystyle (2\pi)^{n}\int_{\mathbb{R}^{m}}\frac{\widehat{\psi}(\underline{\xi})\left[\widehat{\psi}(\underline{\xi})\right]^{\dagger}}{|\underline{\xi}|^{n}}dV(\underline{\xi})<\infty.}
\end{equation}
\end{itemize}
\end{definition}
We say that the function $\psi$ is said to be admissible  and $\mathcal{A}_{\psi}$ its admissibility constant. We notice here also that being admissible  as a mother wavelet  the function $\psi$ should satisfy some oscillation property such as
\begin{equation}
\widehat{\psi}(\underline{0})=0\Longleftrightarrow
\int_{\mathbb{R}^{m}}\psi(\underline{x})dV(\underline{x}).
\end{equation}

For $(a,\underline{b},s)\in\mathbb{R}^{+}\times\mathbb{R}^{m}\times Spin(n)$, we denote
\begin{equation}
\psi^{a,\underline{b},s}(\underline{x})=\frac{1}{a^{\frac{n}{2}}}s\psi(\frac{\overline{s}(\underline{x}-\underline{b})s}{a})\overline{s}.
\end{equation}
It is straightforward that whenever $\psi$ is admissible, the copies $\psi^{a,\underline{b},s}$ are also admissible  and satisfy
\begin{equation}
\mathcal{A}_{\psi^{a,\underline{b},s}}=\frac{a^{n/2}}{(2\pi)^{n}}\mathcal{A}_{\psi}<\infty.
\end{equation}

\begin{proposition}
	The set $\Lambda_{\psi}=\left\{\psi^{a,\underline{b},s}:a>0,\underline{b}\in\mathbb{R}^{m},\,s\in Spin(n)\right\}$ is dense in $L^{2}(\mathbb{R}^{m},\mathbb{R}_{m},dV(\underline{x}))$.
\end{proposition}

\begin{definition}\label{eq:CWT formula}
	The Continuous Clifford wavelet transform  of a function $f$ in $L^{2}(\mathbb{R}^{m},\mathbb{R}_{m},dV(\underline{x}))$  is
\begin{equation}
	T_{\psi}\left[f\right](a,\underline{b},s) =<\psi^{a,\underline{b},s},f>_{L^{2}(\mathbb{R}^{m},\mathbb{R}_{m},dV(\underline{x}))}=\int_{\mathbb{R}^{m}}\left[\psi^{a,\underline{b},s}(\underline{x})\right]^{\dagger}f(\underline{x})dV(\underline{x}).
\end{equation}
\end{definition}

We recall the Parseval-Plancherel type rules which are the most important formula in wavelet theory as they permit to reconstruct functions from their wavelet transforms.

We firstly introduce an inner product relative to the Continuous Clifford wavelet transform. Let
$$
\mathcal{H}_{\psi}=\left\{ T_{\psi}\left[f\right],\;f\in L^{2}(\mathbb{R}^{m},\mathbb{R}_{m},dV(\underline{x}))\right\}
$$
be the image of $L^{2}(\mathbb{R}^{m},\mathbb{R}_{m},dV(\underline{x}))$ relatively to the operator  $T_{\psi}$. The inner product is given by
\begin{equation}
\left[T_{\psi}\left[f\right],T_{\psi}\left[g\right]\right]=\frac{1}{\mathcal{A}_{\psi}}\int\limits _{Spin(n)}\int\limits _{\mathbb{R}^{m}}\int\limits _{\mathbb{R}^{+}}(T_{\psi}\left[f\right](a,\underline{b},s))^{\dagger}T_{\psi}\left[g\right](a,\underline{b},s)\frac{da}{a^{n+1}}dV(\underline{b})ds,
\end{equation}
where $ds$ stands for the Haar measure on $Spin(n)$.
\begin{proposition}\label{prop:CWT isometry}
\begin{enumerate}
\item $T_{\psi}:L^{2}(\mathbb{R}^{m},\mathbb{R}_{m},dV(\underline{x}))\longrightarrow\mathcal{H}_{\psi}$ is an isometry.
 \item	The operator 
$$
T_{\psi}:L^{2}(\mathbb{R}^{m},\mathbb{R}_{m},dV(\underline{x}))\longrightarrow L_{2}(\mathbb{R}_{+}\times\mathbb{R}^{m}\times Spin(n),\dfrac{1}{\mathcal{A}_{\psi}}\dfrac{dadV(\underline{b})ds}{a^{n+1}})
$$
is an isometry.
\item The Parseval-Plancherel equality
\begin{equation}
	\int\limits _{Spin(n)}\int\limits _{\mathbb{R}^{m}}\int\limits _{\mathbb{R}^{+}}(T_{\psi}\left[f\right](a,\underline{b},s))^{2}\frac{da}{a^{n+1}}dV(\underline{b})ds=\mathcal{A}_{\psi}\left\Vert f\right\Vert _{2}^{2}.
\end{equation}
\item The Clifford wavelet  reconstruction formula. For all $f\in L^{2}(\mathbb{R}^{m},\mathbb{R}_{m},dV(\underline{x}))$ we have
\begin{equation}
f(\underline{x})=\frac{1}{A_{\psi}}\int\limits_{Spin(n)}\int\limits_{\mathbb{R}^{m}}\int\limits_{\mathbb{R}^{+}}\psi^{a,\underline{b},s}(\underline{x})T_{\psi}\left[f\right](a,\underline{b},s)\frac{da}{a^{n+1}}dV(\underline{b})ds
\end{equation}
in $L^{2}\left(\mathbb{R}^{m},\mathbb{R}_{m},dV(\underline{x})\right).$
\end{enumerate}
\end{proposition}

\subsection{Monogenic polynomials based Clifford wavelets}
In this section we propose to review a second method to construct wavelets  on Clifford algebras. The idea is based on the so-called monogenic polynomial  which constitute an extension of orthogonal polynomials on Clifford algebras such qs Gegenbauer polynomials  known also as ultra-spheroidal polynomials. Other classes may be developed by the readers by similar techniques. Furthermore, we may refer to \cite{Arfaoui1,Arfaoui2,Arfaoui3, Arfaoui4, Arfaoui-Rezgui-book, Brackx-Schepper-Sommen0,BrackxChisholmSoucek} for other existing classes of polynomials and associated wavelets  in both the classical context and the Clifford one.

Recently a generalized class of Clifford-Gegenbauer polynomials  and wavelets  has been developed in \cite{Arfaoui3} based on the 2-parameters Clifford-weight  function
\begin{equation}
\omega_{\alpha,\beta}(\underline{x})=(1-|\underline{x}|^2)^\alpha (1+|\underline{x}|^2)^\beta.
\end{equation}

Denote $Z_{\ell,m}^{\alpha,\beta}(\underline{x})$ such polynomials and the CK-extension  $F^*$ expressed by
$$
\begin{array}{lll}
F^*(t,\underline{x})&=&
\displaystyle\sum\limits_{\ell=0}^{\infty}\dfrac{t^\ell}{\ell!}Z_{\ell,m}^{\alpha,\beta}(\underline{x})\,\omega_{\alpha-\ell,\beta-\ell}(\underline{x}).
\end{array}
$$
From the monogenicity  relation $(\partial_t+\partial_{\underline{x}})F^*(t,\underline{x})=0$ the authors proved the following result.
\begin{proposition} The 2-parameters Clifford-Gegenbauer Polynomials  $Z_{\ell,m}^{\alpha,\beta}$ satisfy 
\begin{enumerate}
\item the recurrence relation
$$
\begin{array}{lll}
Z_{\ell+1,m}^{\alpha,\beta}(\underline{x})
&=&[2(\alpha-\ell)\underline{x}(1-\underline{x}^2)-2(\beta-\ell)\underline{x}\,(1+\underline{x}^2)] Z_{\ell,m}^{\alpha,\beta}(\underline{x})\\
&&-\,\omega_{1,1}(\underline{x})\partial_{\underline{x}} (Z_{\ell,m}^{\alpha,\beta}(\underline{x})).
\end{array}
$$
\item the Rodriguez formula 
	$$
	Z_{\ell,m}^{\alpha,\beta}(\underline{x})=(-1)^\ell \,\omega_{\ell-\alpha,\ell-\beta}(\underline{x}) \,
	\partial_{\underline{x}}^{\ell} [(1+\underline{x}^2)^\alpha (1-\underline{x}^2)^\beta].
	$$
\end{enumerate}

\end{proposition}

\begin{definition} The generalized 2-parameters Clifford-Gegenbauer mother wavelet is defined by
\begin{equation}
\psi_{\ell,m}^{\alpha,\beta}(\underline{x})
	=Z_{\ell,m}^{\alpha+\ell,\beta+\ell}(\underline{x})
	\omega_{\alpha,\beta}(\underline{x})
	=(-1)^\ell\partial_{\underline{x}}^{(\ell)}\omega_{\alpha+\ell,\beta+\ell}(\underline{x}).
\end{equation}
\end{definition}
Furthermore, the wavelet  $\psi_{\ell,m}^{\alpha,\beta}(\underline{x})$ have vanishing moments as is shown in the next proposition.
\begin{proposition} 
	\begin{enumerate}
		\item Whenever $0<k<-m-\ell-2(\alpha+\beta) $ and $k<\ell$ we have
\begin{equation}
		\displaystyle\int_{\mathbb{R}^m} \underline{x}^k \psi_{\ell,m}^{\alpha,\beta}(\underline{x}) dV(\underline{x})=0.
\end{equation}
\item The Clifford-Fourier transform of $\psi_{\ell,m}^{\alpha,\beta}$ takes the form
\begin{equation}		\widehat{\psi_{\ell,m}^{\alpha,\beta}(\underline{u})}=(-i)^\ell\,\underline{\xi}^\ell(2\pi)^{\frac{m}{2}}\rho^{1-\frac{m}{2}+\ell}\,\displaystyle\int_{0}^\infty\widetilde{\omega}_{\alpha,\beta}^l(r)\,J_{\frac{m}{2}-1}(r\rho)dr,
\end{equation}
where $\widetilde{\omega}_{\alpha,\beta}^l(r)=((1-r^2)\varepsilon_r)^{\alpha+\ell} (1+r^2)^{\beta+\ell} r^{\frac{m}{2}}
$ with $\varepsilon_r=\mbox{sign}(1-r)$ and $J_{\frac{m}{2}-1}$  is the Bessel function of the first kind of order $\frac{m}{2}-1$.
	\end{enumerate}
\end{proposition}

\begin{definition}
	The copy of the generalized 2-parameters Clifford-Gegenbauer wavelet  is defined by
\begin{equation}
	_a^{\underline{b}}\psi_{\ell,m}^{\alpha,\beta}(\underline{x})=a^{-\frac{m}{2}}\psi_{\ell,m}^{\alpha,\beta}(\dfrac{\underline{x}-\underline{b}}{a}).
\end{equation}
\end{definition}
\begin{definition}
	The wavelet transform  of  $f$ in $L_2$ is given by
\begin{equation}
T_\psi(f)(a,\underline{b})=
\displaystyle\int_{\mathbb{R}^m}f(\underline{x})\,_a^{\underline{b}}\psi_{\ell,m}^{\alpha,\beta}(\underline{x})dV(\underline{x}).
\end{equation}
\end{definition}
The following Lemma guarantees that the candidate $\psi_{\ell,m}^{\alpha,\beta}$ is indeed a mother wavelet. 
\begin{lemma}\label{Admissibilityofthenewwavelets}
	The quantity
\begin{equation}
	\mathcal{A}_{\ell,m}^{\alpha,\beta}=\dfrac{1}{\omega_m}\displaystyle\int_{\mathbb{R}^m}|\widehat{\psi_{\ell,m}^{\alpha,\beta}}(\underline{x})|^2 \dfrac{dV(\underline{x})}{|\underline{x}|^m}
\end{equation}
	is finite. ($\omega_m$ is the volume of the unit sphere $S^{m-1}$ in $\mathbb{R}^m$).
\end{lemma}
Consider next the inner product
\begin{equation}
<T_\psi(f)(a,\underline{b}), T_\psi(g)(a,\underline{b})>=\dfrac{1}{\mathcal{A}_{\ell,m}^{\alpha,\beta}}\displaystyle\int_{\mathbb{R}^m}\displaystyle\int_{0}^{+\infty}\overline{T_\psi(f)(a,\underline{b})} T_\psi(f)(a,\underline{b}) \dfrac{da}{a^{m+1}}dV(\underline{b}).
\end{equation}
We obtain the following result known as the Parseval identity analogue.
\begin{theorem}\label{ReconstructionFormula}
	Any function $f\in L_2(\mathbb{R}_m)$ may be reconstructed by
\begin{equation}
	f(x)=\dfrac{1}{\mathcal{A}_{\ell,m}^{\alpha,\beta}}\displaystyle\int_{a>0}\displaystyle\int_{b\in\mathbb{R}^m}T_\psi(f)(a,\underline{b})\,\psi\left(\dfrac{\underline{x}-\underline{b}}{a}\right)\dfrac{da\,dV(\underline{b})}{a^{m+1}},
\end{equation}
	where the equality has to be understood in the $L_2$-sense.
\end{theorem}
For more details readers may refer to (\cite{Arfaoui-Rezgui-book}, \cite{Brackxetal2006}, \cite{Brackx-Schepper-Sommen2}, \cite{Brackxetal2013}, \cite{Brackx2001a},  \cite{Mawardi-Hitzer-1}, \cite{Mawardi-Hitzer-2}).
\section{Donoho-Stark's uncertainty principles for the Clifford wavelet transform}

In this section, we will establish the Donoho–Stark's UPs type for the Clifford wavelet transform $T_{\psi}(f)$ in the case where $f$ and $T_{\psi}(f)$ are close to zero outside measurable sets. For more details, we refer to \cite{Abouelaz}, 

The Donoho–Stark UPs involve the concept of $\epsilon$-concentration. Before we provide the main results of this section we introduce two localization operators with characteristic functions as symbols. their importance is due to the fact that Donoho–Stark hypothesis of $\epsilon$-concentration can be interpreted in terms of the action of operators.

Let $T$ and $\Omega$ be measurable subsets of $\mathbb{R}^m$ and $f$ be a Clifford algebra-valued function of $\mathbb{R}^n$,and $T_{\psi}(f)$
be (if exists) its Clifford Wavelet transform.
The first operator is the time-limiting operator $P_T$ given by
\begin{equation}
(P_Tf)(\underline{x})\overset{def} {=}(\chi_T f)(\underline{x})=
\begin{cases}
f(\underline{x}), \quad if \quad \underline{x}\in T\\
0 \quad otherwise.
\end{cases}
\end{equation}
This operator delete the part of $f$ outside $T$.
The second operator is the frequency-limiting operator $Q_{\Omega}$ defined by
\begin{equation}
T_{\psi}\left[Q_{\Omega} f\right](a,\underline{b})\overset{def} {=}(\chi_{\Omega} T_{\psi})(a,\underline{b})=
\displaystyle\int_{\mathbb{R}^m}f(\underline{x})\,\psi_{a,\underline{b}}(\underline{x})dV(\underline{x}).
\end{equation}
Which means $Q_{\Omega} f$ is a partial reconstruction of $f$ using only frequency information from frequencies in $\Omega$, and $T_{\psi}\left[Q_{\Omega} f\right]$ vanishes outside $\Omega$.

\begin{proposition} Let $T$ and $\Omega$ be measurable subsets of $\mathbb{R}^m$, then if $f\in L^2(\mathbb{R}^m, \mathbb{R}_m)$, we have 
\begin{equation}
 || T_{\psi}(Q_{\Omega}P_T f)||_{L^2(\mathbb{R}^m, \frac{da dV(\underline{b})}{a^{m+1}})}\leq  ||P_T f||_2\,\phi_{\Omega,T}(\psi). 
\end{equation}
\end{proposition}
\begin{proof}
We assume that $|T|<\infty$ and $|\Omega|<\infty$, then
$$
T_{\psi}\left[Q_{\Omega} f\right](a,\underline{b})\overset{def} {=}(\chi_{\Omega} T_{\psi}(f))(a,\underline{b})
$$
$$
\begin{array}{lll}
|| T_{\psi}(Q_{\Omega}P_T f)||^2_{L^2(\mathbb{R}^m, \frac{da dV(\underline{b})}{a^{m+1}})}
&=& \left( \displaystyle\int_{\mathbb{R}^m}|\chi_{\Omega}(a,\underline{b})|^2 |T_{\psi}(P_T f)(a,\underline{b})|^2 \frac{da dV(\underline{b})}{a^{m+1}}\right)^\frac{1}{2}\\
&=& \left(\displaystyle\int_{\Omega} |T_{\psi}(P_Tf)(a,\underline{b})|^2 \frac{da dV(\underline{b})}{a^{m+1}}\right)^\frac{1}{2}.
\end{array}
$$
Now, observe that 
$$
T_{\psi}(P_Tf)(a,\underline{b})=\displaystyle\int_{\mathbb{R}^m} (P_T f) (\underline{x}) \dfrac{1}{a^{\frac{m}{2}}} \psi (\frac{\underline{x}-\underline{b}}{a}) dV(\underline{x})= \displaystyle\int_{T} f (\underline{x}) \dfrac{1}{a^{\frac{m}{2}}} \psi (\frac{\underline{x}-\underline{b}}{a}) dV(\underline{x}).
$$
By Holder inequality, it follows that 
$$
\begin{array}{lll}
|T_{\psi}(P_Tf)(a,\underline{b})|
&\leq & \left(\displaystyle\int_{T} |f(\underline{x})|^2 dV(\underline{x}) \right)^{\frac{1}{2}}  \left( \displaystyle\int_{T} |\dfrac{1}{a^{\frac{m}{2}}} \psi (\frac{\underline{x}-\underline{b}}{a})|^2 dV(\underline{x}) \right)^{\frac{1}{2}}\\
&\leq & ||P_T f||_2 \left(\displaystyle\int_{T_{a,b}} | \psi (t)|^2 dV(t) \right)^{\frac{1}{2}}, 
 \end{array}
$$
with $T_{a,b}= aT+b$. Then, 
$$
|| T_{\psi}(Q_{\Omega}P_T f)||^2_{L^2(\mathbb{R}^m, \frac{da dV(\underline{b})}{a^{m+1}})}\leq  ||P_T f||_2^2 \displaystyle\int_{\Omega}  \left(\displaystyle\int_{T_{a,b}} |\psi (t)|^2 dt \right) \frac{da dV(\underline{b})}{a^{m+1}}\leq  ||P_T f||_2^2 \phi_{\Omega,T},
$$
where $\phi_{\Omega,T}=\displaystyle\int_{\Omega}  \left(\displaystyle\int_{T_{a,b}} |\psi (t)|^2 dV(t) \right) \frac{da dV(\underline{b})}{a^{m+1}}$. This yields the desired result.
\end{proof}

\begin{corollary}
Let $\Omega=[\alpha, +\infty) \times \tilde{\Omega} $ be a band with $\tilde{\Omega}$, finite Lebesgue measure, then
\begin{equation}
|| T_{\psi}(Q_{\Omega}P_T f)||_{L^2(\mathbb{R}^m, \frac{da dV(\underline{b})}{a^{m+1}})}\leq  ||P_T f||_2  || \psi||_2 |\tilde{\Omega}|^{\frac{1}{2}}.
\end{equation}
\end{corollary}

\begin{lemma}[Parseval-Plancherel equality]For the function $f\in L^2(\mathbb{R}^m, \mathbb{R}_m, dV(\underline{x}))$, 
\begin{equation}\label{Parseva's Identity} 
|| T_{\psi}(f)||_{L^2(\mathbb{R}^m, \mathbb{R}_m,  dV(\underline{x}))}=A_\psi ||f||_{L^2(\mathbb{R}^m, \mathbb{R}_m,  dV(\underline{x}))}, 
\end{equation}
\end{lemma}

\begin{theorem} (Donoho-Stark’s uncertainty principle type) Let $T$ and $\Omega$ be measurable subsets of $\mathbb{R}^m$,
and $f\in L^1\cap L^2(\mathbb{R}^m, \mathbb{R}_m)$. If $f$ is $\epsilon_T$-concentrated on $T$ and $T_{\psi}(f)$ is $\epsilon_\Omega$-concentrated on $\Omega$, then
\begin{equation}\label{Donoho-Stark’s uncertainty principle type}
 || T_{\psi}(f)||_{L^2(\mathbb{R}_+\times\mathbb{R}^m , dV (\underline{x}))}\leq \dfrac{[ A_\psi \epsilon_T + \phi_{\Omega,T}(\psi )]}{(1-\epsilon_{\Omega})} ||f||_{L^2(\mathbb{R}^m, dV(\underline{x}))}.  
\end{equation}
\end{theorem}
\begin{proof}
We assume that $T$ and $\Omega$ have  finite measures. Then, we have
$$
||f-P_Tf||_2 \leq \epsilon_T ||f||_2.
$$ 
As $T_{\psi}(f)$ is $\epsilon_\Omega$-concentrated, we obtain 
\begin{equation}\label{1}
||T_{\psi}(f-Q_\Omega f)||_{2} \leq \epsilon_\Omega ||T_{\psi}( f)||_{2}.
\end{equation}
Now observe that,  
\begin{equation}\label{2}
||T_{\psi}(Q_\Omega f-Q_\Omega P_T f)||_2 \leq 
||T_{\psi}( f-P_T f)||_2 = A_{\psi} || f-P_T f||_2.
\end{equation}

Consequently, by using (\ref{1}), (\ref{2}), we obtain
$$
\begin{array}{lll}
||T_{\psi}(f-Q_\Omega P_T f)||_2 &\leq& ||T_{\psi}(f-Q_\Omega f)||_2+  
||T_{\psi}(Q_\Omega f-Q_\Omega P_T f)||_2\\
 &\leq& \epsilon_{\Omega} ||T_\psi(f)||_2+ A_\psi ||f-P_Tf||_2\\
  &\leq& \epsilon_{\Omega} ||T_\psi(f)||_2+ A_\psi \epsilon_T ||f||_2.
\end{array}
$$
Otherwise, we have
$$
\begin{array}{lll}
||T_{\psi}(f)||_2 &\leq& ||T_{\psi}(f-Q_\Omega P_T f)||_2+  
||T_{\psi}(Q_\Omega P_T f)||_2\\
&\leq& \epsilon_{\Omega} ||T_\psi(f)||_2+ A_\psi \epsilon_T ||f||_2+ \phi_{\Omega,T}(\psi )||f||_2.
\end{array}
$$
Therefore, 
$$
(1-\epsilon_{\Omega})||T_{\psi}(f)||_2 \leq[ A_\psi \epsilon_T + \phi_{\Omega,T}(\psi )]||f||_2.
$$
A a result,
$$
||T_{\psi}(f)||_2 \leq \dfrac{[ A_\psi \epsilon_T + \phi_{\Omega,T}(\psi )]}{(1-\epsilon_{\Omega})} ||f||_2. 
$$
\end{proof}

\begin{corollary}
Suppose that $f\in L^2(\mathbb{R}^m, \mathbb{R}_m)$, with $f\neq 0 $,  $f$ is $\epsilon_T$-concentrated on $T$, and $T_{\psi}$ is $\epsilon_{\Omega}$-concentrated on $\Omega$. Then, 
\begin{enumerate}
\item \begin{equation}\label{cor4.2}
    (1-\epsilon_{\Omega}-\epsilon_{T}) A_\psi \leq \phi_{\Omega,T}(\psi ).
\end{equation}
\item Whenever $Supp f\subseteq T$ and $Supp T_{\psi}\subseteq \Omega$,  we have 
\begin{equation}
A_\psi\leq \phi_{\Omega,T}(\psi ).
\end{equation}
\end{enumerate}
\end{corollary}
\begin{proof}
\begin{enumerate}
\item The first assertion is follows from the Parseval's idendity (\ref{Parseva's Identity}) and the Donoho-Stark’s uncertainty principle type (\ref{Donoho-Stark’s uncertainty principle type}).
\item The second assertion is a natural consequence of (\ref{cor4.2}) by choosing $\epsilon_{T}=\epsilon_{\Omega}=0$.
\end{enumerate}
\end{proof}
\section{Conclusion}
In this paper, the Donoho-Stark's uncertainty principle associated to wavelet transforms in the framework of Clifford analysis has been established.
\section*{Data Availability Statement}
My manuscript has no associate data.


\begin{thebibliography}{00}
\bibitem{Abouelaz} Abouelaz, A., Achak, A., Daher, R. et al. Donoho–Stark’s uncertainty principle for the quaternion Fourier transform. Bol. Soc. Mat. Mex. 26, 587-597 (2020)

\bibitem{Arfaoui1} S. Arfaoui, A. Ben Mabrouk and C. Cattani, New type of Gegenbauer-Hermite monogenic polynomials and associated Clifford wavelets. Journal of Mathematical Imaging and Vision, 62(1)  73-97 (2020)

\bibitem{Arfaoui2} S. Arfaoui, A. Ben Mabrouk and C. Cattani, New type of Gegenbauer-Jacobi-Hermite monogenic polynomials and associated continuous Clifford wavelet transform, Acta Applicandae Mathematicae, 1-35, (2020)

\bibitem{Arfaoui3} S. Arfaoui and A. Ben Mabrouk, Some old orthogonal polynomials revisited and associated wavelets: two-parameters Clifford-Jacobi polynomials and associated spheroidal wavelets. Acta Applicandae Mathematicae 155(1) 177–195 (2018).
 
\bibitem{Arfaoui4} S. Arfaoui and A. Ben Mabrouk, Some ultraspheroidal monogenic Clifford Gegenbauer Jacobi polynomials and associated wavelets. Advances in Applied Clifford Algebras 27: 2287. (2017) 
 
\bibitem{Arfaoui-Rezgui-book} S. Arfaoui, I. Rezgui and A. Ben Mabrouk, Wavelet Analysis on the Sphere: Spheroidal Wavelets. Walter de Gruyter (March 20, 2017), ISBN-10:311048109X, ISBN-13: 978-3110481099.

\bibitem{Arfaouietal21BookWavelet}  Arfaoui S, Ben Mabrouk A and Cattani C (2021)  Wavelet analysis Basic concepts and applications,  CRC Taylor-Francis, Chapmann  Hall, Boca Raton,  1st Ed., April 21, 2021.

\bibitem{Banouh} Banouh, H., Ben Mabrouk, A. \& Kesri, M. Clifford Wavelet Transform and the Uncertainty Principle. Adv. Appl. Clifford Algebras 29, 106 (2019). 

\bibitem{Banouh1} Banouh, H. and Ben Mabrouk, A. A sharp Clifford wavelet Heisenberg-type uncertainty principle. J. Math. Phys. 61, 093502 (2020)

\bibitem{BrackxChisholmSoucek} Brackx, F.,  Chisholm J.S.R. and Soucek, V.:Clifford Analysis and Its Applications. NATO Science Series, Series II: Mathematics, Physics and Chemistry - Vol. 25. Springer (2000)

\bibitem{Brackx-Schepper-Sommen0} Brackx, F., Delanghe, R. and Sommen, F.: Clifford analysis, Pitman Publication, (1982)

\bibitem{Brackxetal2006} Brackx, F.,  De Schepper, N. and Sommen, F.: The Two-Dimensional Clifford Fourier Transform. J. Math. Imaging, 26, 5-18 (2006)

\bibitem{Brackx-Schepper-Sommen2} Brackx, F.,  De Schepper, N. and Sommen, F.: The Fourier Transform in Clifford analysis. Advances in Imaging and Electron Physics, 156, 55-201 (2009)

\bibitem{Brackxetal2013} Brackx, F., Hitzer, E., Sangwine, S.J.: History of quaternion and Clifford-fourier transforms and wavelets, in Quaternion and Clifford fourier transforms and wavelets. In Trends in Mathematics 27. p.XI-XXVII (2013)

\bibitem{Brackx2001a} Brackx, F., Sommen, F.:  The continuous wavelet transform in Clifford analysis.   Clifford Analysis and Its Applications, pages 9-26. Springer Netherlands, (2001)
	
\bibitem{Cartan1966} Cartan, E.:  The theory of spinors. Courier Corporation (1966)
	
\bibitem{Dahkleteal} Dahkle, S., Kutyniok, G., Maass, P., Sagiv, C., Stark, H.-G., Teschke, G.: The uncertainty principle associated with the continuous shearlet transform. International Journal of Wavelets, Multiresolution and Information Processing, 6(2), 157-181 (2008)

\bibitem{DeBie1} De Bie, H.: Clifford algebras, Fourier transforms and quantum mechanics, arXiv: 1209.6434v1, 39 pages (2012)

\bibitem{DeBie2} De Bie, H., Xu, Y.: On the Clifford Fourier transform. ArXiv: 1003.0689, 30 pages (2010)

\bibitem{Donoho} Donoho, D.L., Stark, P.B.: Uncertainty principles and signal recovery. SIAM J. Appl. Math. 49(3), 906-931 (1989)

\bibitem{ElHaouietal} El Haoui, Y., Fahlaoui, S., Hitzer, E.: Generalized Uncertainty Principles associated with the Quaternionic Offset Linear Canonical Transform. arXiv:1807.04068v2 [math.CA], 15 pages (2019)

\bibitem{ElHaoui-Fahlaoui} El Haoui, Y., Fahlaoui, S.: Donoho-Stark's Uncertainty Principles in Real Clifford Algebras. arXiv:1902.08465v1 [math.CA], 9 pages (2019)

\bibitem{Fu2015} Fu, Y., Li, L.: Uncertainty principle for multivector-valued functions.  International Journal of Wavelets, Multiresolution and Information Processing, 13(01), 1-8 (2015)

\bibitem{Hamilton1866} Hamilton, W.R.:  Elements of quaternions.  Longmans, Green,  Company (1866)
	
\bibitem{Hitzer2} Hitzer, E.: Directional Uncertainty Principle for Quaternion Fourier Transform. Adv. appl. Clifford alg. 20, 271-284. (2010)

\bibitem{Hitzer-Mawardi-1} Hitzer, E., Mawardi, B.: Uncertainty Principle for the Clifford-Geometric Algebra $\mathcal{Cl}_{3,0}$ based on Clifford Fourier Transform. arXiv:1306.2089v1 [math.RA] 4 pages (2013)

\bibitem{Jday2018} Jday, R.: \newblock Heisenberg's and hardy's uncertainty principles in real clifford algebras. \newblock {\em Integral Transforms and Special Functions} 29(8), 663-677 (2018)

\bibitem{Kouetal} Kou, K.I., Ou, J.-Y., Morais, J.: On Uncertainty Principle for Quaternionic Linear Canonical Transform. Abstract and Applied Analysis, Article ID 725952, 14 pages, (2013)

\bibitem{Mawardi-Ryuichi} Mawardi, B., Ryuichi, A.: A Simplified Proof of Uncertainty Principle for Quaternion Linear Canonical Transform. Abstract and Applied Analysis, 11 pages. 

\bibitem{Mawardi-Ryuichi-1} Mawardi, B., Ashino, R.: Logarithmic uncertainty principle for quaternion linear canonical transform. Proceedings of the 2016 International Conference on Wavelet Analysis and Pattern Recognition, Jeju, South Korea, 10-13 July, 6 pages (2016)

\bibitem{Mawardi-Ryuichi-2} Mawardi, B., Ashino, R.: A Variation on Uncertainty Principle and Logarithmic Uncertainty Principle for Continuous Quaternion Wavelet
Transforms. Abstract and Applied Analysis, Article ID 3795120, 11 pages, (2017)

\bibitem{Mawardi-Hitzer-1} Mawardi, B., Hitzer, E.: Clifford Algebra $Cl(3,0)$-valued Wavelets and Uncertainty Inequality for Clifford Gabor Wavelet Transformation, Preprints of Meeting of the Japan Society for Industrial and Applied Mathematics, ISSN: 1345-3378, Tsukuba University, 16-18 Sep. 2006, Tsukuba, Japan, pp. 64-65 (2006)

\bibitem{Mawardi-Hitzer-2} Mawardi, B., Hitzer, E.: Clifford algebra $Cl(3,0)$-valued wavelet transformation, Clifford wavelet uncertainty inequality and Clifford Gabor wavelets. International Journal of Wavelets, Multiresolution and Information Processing 5(6), 997-1019 (2007)

\bibitem{Mawardi-Hitzer-3} Mawardi, B., Hitzer, E.: Clifford Fourier Transformation and Uncertainty Principle for the Clifford Geometric Algebra $Cl_{3,0}$. Adv. appl. Clifford alg. 16, 41-61. (2006)

\bibitem{Mawardi-Hitzer-4} Mawardi, B., Hitzer, E.: Clifford Fourier Transform on Multivector Fields and Uncertainty Principles for Dimensions $n=2(mod4)$ and $n=3(mod4)$. Adv. appl. Clifford alg. 18, 715-736.  (2008)

\bibitem{Mawardi-Hitzer-Hayashi-Ashino} Mawardi, B., Hitzer, E., Hayashi, A., Ashino, R.: An Uncertainty Principle for Quaternion Fourier Transform, Computer \& Mathematics with Applications 56, 2398-2410 (2008)

\bibitem{Mejjaolietal} Mejjaoli, H., Ben Hamadi, N., Omri, S.: Localization operators, time frequency concentration and quantitative-type uncertainty for the continuous wavelet transform associated with spherical mean operator. International Journal of Wavelets, Multiresolution and Information ProcessingVol. 17(04), ID 1950022, 28 pages.  (2019)


\bibitem{Msehli-Rachdi-1} Msehli, N., Rachdi, L.T.: Heisenberg-Pauli-Weyl Uncertainty Principle for the Spherical Mean Operator. Mediterranean Journal of Mathematics, 7(2), 169-194. (2010)

\bibitem{Msehli-Rachdi-2} Msehli, N., Rachdi, L.T.: Beurling-H\"ormander uncertainty principle for the spherical mean operator 10(2), Article 38, 22 pages (2009)
\bibitem{Rachdi-Meherzi} Rachdi, L.T., Meherzi, F.: Continuous Wavelet Transform and Uncertainty Principle Related to the Spherical Mean Operator. Mediterranean Journal of Mathematics, 14(1). (2016)

\bibitem{Rachdi-Amri-Hammami} Rachdi, L.T., Amri, B., Hammami, A.: Uncertainty principles and time frequency analysis related to the Riemann–Liouville operator. Annali Dell'Universita' Di Ferrara.  (2018)

\bibitem{Rachdi-Herch} Rachdi, L.T., Herch, H.: Uncertainty principles for continuous wavelet transforms related to the Riemann–Liouville operator. Ricerche Di Matematica, 66(2), 553-578, (2017)

\bibitem{Sen} Sen, D.: The uncertainty relations in quantum mechanics. Current Science 107(2), 203-218 (2018)

\bibitem{Soltani} Soltani, F.: $L^P$ Donoho–Stark Uncertainty Principles for the Dunkl Transform on Rd. J. Phys.Math. 05, 127 (2014). 
\bibitem{Yangetal1} Yang, Y., Dang, P., Qian, T.: Stronger uncertainty principles for hypercomplex signals, Complex Variables and Elliptic Equations, 60(12), 1696-1711, (2015)


\end{thebibliography}
\end{document}